\documentclass[12pt,twoside]{article}
\usepackage{amsmath,amssymb,theorem}
\usepackage{graphicx}
\usepackage{amssymb}

\usepackage[dvips]{epsfig}

\newtheorem{theorem}{Theorem}

\newtheorem{lemma}{Lemma}
\newtheorem{proposition}{Proposition}

\newtheorem{remark}{Remark}
\newenvironment{proof}{{\bf Proof}.\ }{ \hfill $\square$}
\newcommand{\R}{\mathbb{R}}

\begin{document}

\title{Sharp bounds on enstrophy growth \\
in the viscous Burgers equation}

\author{Dmitry Pelinovsky \\
{\small Department of Mathematics, McMaster
University, Hamilton, Ontario, Canada, L8S 4K1}}

\date{\today}
\maketitle

\begin{abstract}
We use the Cole--Hopf transformation and the Laplace method for the heat equation to justify
the numerical results on enstrophy growth in the viscous Burgers equation on the unit circle.
We show that the maximum enstrophy achieved in the time evolution is scaled as $\mathcal{E}^{3/2}$,
where $\mathcal{E}$ is the large initial enstrophy, whereas the time needed for reaching the
maximal enstrophy is scaled as $\mathcal{E}^{-1/2}$. These bounds are
sharp for sufficiently smooth initial conditions.
\end{abstract}

\vspace{1cm}

Existence and regularity of solutions of the three-dimensional Navier--Stokes equations
is a challenging problem that attracted recently many researchers \cite{Doering,Lorenz}.
One possibility to continue local solutions globally in time is to control enstrophy
of the Navier--Stokes equations during the time evolution. Enstrophy is an integral quantity in
space, which may diverge in a finite evolution time, indicating singularities
in the Navier--Stokes equations. To control the growth of enstrophy, Lu and Doering \cite{LuLu}
studied bounds on the instantaneous growth rate of enstrophy and showed numerically that
these bounds are sharp in the limit of large enstrophy. However, these bounds
on the instantaneous growth rate do not imply that enstrophy blows up in a finite time, because solutions
of the Cauchy problem associated with the Navier-Stokes equation may deviate away from
the maximizers of the bounds even if the initial data are close to the maximizers.

To deal with this problem, Ayala and Protas \cite{Diego} looked at a toy model,
the one-dimensional viscous Burgers equation. Their numerical results
indicated that the bounds on the instantaneous growth rate are not sharp when
they are integrated on a finite time interval in conjugation with the energy dissipation
and the Poincar\'e inequality. Limited accuracy of numerical results did not allow them
to conclude if better estimates on the enstrophy growth over a finite time interval can
be justified within this context.

To address this question, we consider the Cauchy problem
for the one-dimensional viscous Burgers equation \cite{Burgers},
\begin{equation}
\label{Burgers}
\left\{ \begin{array}{l} u_t + 2u u_x = u_{xx}, \;\;\; x \in \mathbb{T}, \; t \in \mathbb{R}_+,\\
u|_{t = 0} = u_0, \phantom{texttext} x \in \mathbb{T},\end{array} \right.
\end{equation}
where $\mathbb{T} = \left[-\frac{1}{2},\frac{1}{2}\right]$ is the unit circle
equipped with the periodic boundary conditions for the real-valued function $u$.
Local well-posedness of the initial-value problem (\ref{Burgers}) holds for
$u_0 \in H^s_{\rm per}(\mathbb{T})$ with $s > -\frac{1}{2}$ \cite{Dix}.
Global existence holds in $H^s_{\rm per}(\mathbb{T})$ for any integer $s \geq 0$,
and we consider here global solutions of the viscous Burgers
equation (\ref{Burgers}) in $H^1_{\rm per}(\mathbb{T})$. For simplicity,
we will assume that $u_0$ is odd in $x$, which implies that $u(-x,t) = -u(x,t)$
holds for all $t \geq 0$ and $x \in \mathbb{T}$.

All $L^p_{\rm per}$ norms with even $p$ are monotonically decaying
in the time evolution of the viscous Burgers equation (\ref{Burgers}).
In particular, the energy dissipation equation follows from (\ref{Burgers})
after integration by parts:
\begin{equation}
\label{energy}
K(u) = \frac{1}{2} \int_{\mathbb{T}} u^2 dx \quad \Rightarrow \quad
\frac{d K(u)}{d t} = \int_{\mathbb{T}} u (u_{xx} - 2 u u_x) dx = - 2 E(u),
\end{equation}
where $E(u) = \frac{1}{2} \int_{\mathbb{T}} u_x^2 dx$ is a positive definite enstrophy.
For a smooth solution $u \in C(\mathbb{R}_+,H^3_{\rm per}(\mathbb{T}))$, the enstrophy
changes according to the equation:
\begin{eqnarray}
\label{rate-of-change-E}
\frac{d E(u)}{d t} = \int_{\mathbb{T}} u_x (u_{xxx} - 2 u u_{xx} - 2 u_x^2 ) dx
= -\int_{\mathbb{T}} ( u_{xx}^2 + u_x^3 ) dx \equiv R(u),
\end{eqnarray}
where $R(u)$ is the rate of change of $E(u)$. We can see from (\ref{rate-of-change-E})
that $R(u)$ is a sum of negative definite
quadratic part and a sign-indefinite cubic part. The quadratic part corresponds to the
diffusion term of the viscous Burgers equation and the cubic part corresponds to
the nonlinear advection term. It is the latter term that may lead to the enstrophy growth during the initial
time evolution.

Lu and Doering \cite{LuLu} showed that the rate of change $R(u)$ in (\ref{rate-of-change-E}) can be estimated by
\begin{equation}
\label{bound-on-R}
R(u) \leq \frac{3}{2} E^{5/3}(u),
\end{equation}
and illustrated numerically that the growth $R(u) = \mathcal{O}(\mathcal{E}^{5/3})$
is achieved in the limit of large $\mathcal{E} := E(u)$.
If bound (\ref{bound-on-R}) is sharp on the time interval $[0,T]$
for some $T > 0$, then integration of the enstrophy equation (\ref{rate-of-change-E})
with the use of the energy dissipation equation (\ref{energy}) implies
\begin{equation}
\label{nonlocal-bound}
E^{1/3}(u(T)) - E^{1/3}(u_0) \leq \frac{1}{2} \int_0^T E(u(t)) dt
= \frac{1}{4} \left[ K(u_0) - K(u(T)) \right].
\end{equation}

Using the Poincar\'e inequality for periodic functions with zero mean,
\begin{equation}
\label{Poincare-inequality}
K(u_0) \leq \frac{1}{4 \pi^2} E(u_0),
\end{equation}
and neglecting $K(u(T))$ in (\ref{nonlocal-bound}), we can obtain
\begin{equation}
\label{bound-on-E}
E(u(T)) \leq \left(  \mathcal{E}^{1/3} + \frac{1}{16 \pi^2}  \mathcal{E} \right)^3, \quad \mathcal{E} := E(u_0).
\end{equation}
Ayala and Protas \cite{Diego} showed numerically that the integral bound (\ref{bound-on-E}) is not sharp
even in the limit of large ${\cal E}$. Instead, they obtained the following numerical result:
\begin{equation}
\label{future-rates}
T_* = {\cal O}({\cal E}^{-1/2}), \quad E(u(T_*))
= {\cal O}({\cal E}^{3/2}), \quad K(u(T_*)) = {\cal O}({\cal E}),
\end{equation}
where $T_*$ is the value of time $t$, at which $E(u(t))$ is maximal. They also wanted to show that
$K(u_0) - K(u(T_*)) = {\cal O}({\cal E}^{1/2})$, so that the full integral bound (\ref{nonlocal-bound})
could be useful but numerical approximations of this quantity suffered from large errrors:
\begin{equation}
\label{suspicious-bound}
K(u_0) - K(u(T_*)) = \mathcal{O}({\cal E}^{0.68 \pm 0.25}).
\end{equation}

In the previous work \cite{Pelinovsky}, we used dynamical system methods to study analytically the enstrophy
growth in the viscous Burgers equation. Our technique based on the self-similar transformation
and analysis of asymptotic stability of viscous shocks in an unbounded domain did not rely on
the remarkable properties of the viscous Burgers equation such as
the Cole--Hopf transformation \cite{Cole,Hopf} of equation (\ref{Burgers})  to the linear heat equation.
On the other hand, a weaker version of bounds (\ref{future-rates}) modified by logarithmic factors
was justified as a result of this approach. No estimate on $K(u_0) - K(u(T_*))$ has been obtained.

In this work, we shall rely on the Cole--Hopf transformation and use the Laplace method for the heat equation.
The Laplace method is typically used to recover solutions of the inviscid Burgers equations from
solutions of the viscous Burgers equation in the limit of vanishing viscosity
(see, e.g., \cite[Chapter 2]{Whitham}, \cite[Example 6.5.2]{AblowitzFokas}, or \cite[Section 3.6]{Miller}). 
Applications of this method to statistical properties of the Burgers turbulence can be 
found in \cite{Frach}.

Note that the limit of vanishing viscosity corresponds to the limit of large
enstrophy in the context of our work. We shall implement the Laplace method
to justify numerical results (\ref{future-rates}) and to estimate
$K(u_0) - K(u(T_*))$ as $\mathcal{E} \to \infty$.
Our main result is the following theorem.

\begin{theorem}
Consider the initial-value problem (\ref{Burgers}) with
initial data $u_0(x) = k f(x)$, where $f \in C^3_{\rm per}(\mathbb{T})$ is an arbitrary odd function
such that $f''(x) \geq 0$ for all $x \in \left[0,\frac{1}{2}\right]$.
Consider the limit $k \to \infty$ and denote the initial enstrophy by
$\mathcal{E} = E(u_0) = \mathcal{O}(k^2)$. There exists $T_* > 0$ such that
the enstrophy $E(u)$ achieves its maximum at $u_* = u(\cdot,T_*)$ with
\begin{equation}
T_* = \mathcal{O}(\mathcal{E}^{-1/2}), \quad
E(u_*) = \mathcal{O}(\mathcal{E}^{3/2}), \quad
K(u_*) = \mathcal{O}(\mathcal{E}),
\label{arg-max-theorem}
\end{equation}
and
\begin{equation}
\label{sharp-bound-K}
K(u_0) - K(u_*) = \mathcal{O}(\mathcal{E}),
\end{equation}
where all bounds are sharp as $\mathcal{E} \to \infty$.
\label{theorem-time}
\end{theorem}

Because $f \in C^3_{\rm per}(\mathbb{T})$ is an odd function with $f''(x) \geq 0$ for all $x \in \left[0,\frac{1}{2}\right]$,
it implies necessarily that $f(0) = f\left(\frac{1}{2}\right) = 0$,
$f(x) < 0$ for all $x \in \left(0,\frac{1}{2}\right)$, and
$f'(y)$ is a monotonically increasing function
from $f'(0) < 0$ to $f'\left(\frac{1}{2}\right) > 0$ with a unique zero at $x_*$ such that
$f'(x_*) = 0$. We will actually show that
\begin{equation}
\label{bound-1}
T_* = \frac{x_*}{2 k |f(x_*)|}, \quad
E(u_*) = \frac{1}{2} k^3 |f(x_*)|^3 + \mathcal{O}(k^2), \quad \mbox{\rm as} \quad k \to \infty,
\end{equation}
and
\begin{equation}
\label{bound-2}
K(u_0) - K(u_*) = k^2 \left( \int_{0}^{x_*} f^2(y) dy - \frac{1}{3} x_* f^2(x_*) \right) + \mathcal{O}(k),
\quad \mbox{\rm as} \quad k \to \infty,
\end{equation}
where the leading-order terms are all nonzero. Sharp bounds (\ref{bound-1})--(\ref{bound-2}) rule out
the hope of using the integral bound (\ref{nonlocal-bound}) that follows from the instantaneous
estimate (\ref{bound-on-R}) and the balance equations (\ref{energy}) and (\ref{rate-of-change-E}).

As an example, we can consider $f(x) = -2\pi \sin(2\pi x)$, which saturates
the Poincar\'{e} inequality (\ref{Poincare-inequality}) and satisfies the conditions
of Theorem \ref{theorem-time}. Then, $x_* = \frac{1}{4}$, where $f'(x_*) = 0$, and the sharp
bounds (\ref{bound-1})--(\ref{bound-2}) yield the explicit expressions
\begin{equation*}
T_* = \frac{1}{16 \pi k}, \quad
E(u_*) = 4 \pi^3 k^3 + \mathcal{O}(k^2), \quad
K(u_0) - K(u_*) = \frac{\pi^2}{6} k^2 + \mathcal{O}(k),
\quad \mbox{\rm as} \quad k \to \infty.
\end{equation*}

A generalization of Theorem \ref{theorem-time} can be developed
for any $f \in C^3_{\rm per}(\mathbb{T})$ with
finitely many changes in the sign of $f''$ on $\left[0,\frac{1}{2}\right]$
by the price of lengthier technical computations.
On the other hand, it is not clear if the bounds (\ref{arg-max-theorem}) and (\ref{sharp-bound-K})
remain sharp for initial conditions with limited regularity.

In the remainder of this paper, we shall prove Theorem \ref{theorem-time}.
Using the Cole--Hopf transformation \cite{Cole,Hopf},
\begin{equation}
\label{Cole-Hopf}
u(x,t) = - \frac{\partial}{\partial x} \log \psi(x,t), \quad
\psi(x,t) > 0 \quad \mbox{\rm for all } \; (x,t),
\end{equation}
we rewrite the Cauchy problem (\ref{Burgers}) in the equivalent form
\begin{equation}
\label{heat}
\left\{ \begin{array}{l} \psi_t = \psi_{xx}, \quad \quad \quad  x \in \mathbb{T}, \;\;\; t \in \mathbb{R}_+,\\
\psi|_{t = 0} = \psi_0, \phantom{textt} x \in \mathbb{T},\end{array} \right.
\end{equation}
where
$$
\psi_0(x) = e^{-\int_0^x u_0(s) ds} = e^{- k F(x)}, \quad F(x) := \int_0^x f(s) ds.
$$
If $f \in C^3_{\rm per}(\mathbb{T})$ is odd in $x$, then
$\psi_0 \in C^4_{\rm per}(\mathbb{T})$ is even in $x$.

Although the Cauchy problem (\ref{heat}) is posed on the periodic domain, we can still
construct the solution as a convolution of the initial data $\psi_0$ with
the heat kernel $G_t : \R \to \R_+$ defined on the entire axis,
\begin{equation}
G_t(x) = \frac{1}{\sqrt{4 \pi t}} e^{-\frac{x^2}{4t}}, \quad x \in \R, \;\; t \in \R_+.
\end{equation}
This is justified by the generalized Young inequality,
$$
\| G_t \star \psi_0 \|_{L^r(\mathbb{R})} \leq \| G_t \|_{L^p(\mathbb{R})} \| \psi_0 \|_{L^q(\mathbb{R})},
\quad 1 \leq p,q,r \leq \infty : \quad \frac{1}{p} + \frac{1}{q} = 1 + \frac{1}{r},
$$
where $\star$ denotes the convolution integral. If $\psi_0 \in L^{\infty}_{\rm per}(\mathbb{R})$,
then $G_t \star \psi_0 \in L^{\infty}_{\rm per}(\mathbb{R})$ for all $t \in \R_+$.
Therefore, we write the solution to the Cauchy problem (\ref{heat}) in the explicit form,
\begin{equation}
\label{solution-heat}
\psi(x,t) = \frac{1}{\sqrt{4 \pi t}} \int_{-\infty}^{\infty} e^{- k F(y) -\frac{(x-y)^2}{4t}} dy.
\end{equation}

Let us define the parametrization of the time variable by $t = \frac{1}{2 k a}$, where
$a \in \R_+$ is a new parameter. The solution of the viscous Burgers equation can now be written
in the explicit form,
\begin{equation}
\label{solution-Burgers}
u(x,t) = - \frac{\partial}{\partial x} \log I_{x,a}(k), \quad
I_{x,a}(k) := \int_{-\infty}^{\infty} e^{- k \phi_{x,a}(y)} dy,
\end{equation}
where $\phi_{x,a}(y) := F(y) + \frac{1}{2} a (x-y)^2$. Integral $I_{x,a}(k)$ can be studied
in the limit $k \to \infty$ by means of the Laplace method (Section 3.4 in \cite{Miller}).
We recall the main result of the Laplace method.

\begin{proposition}
For any $-\infty \leq a < b \leq \infty$,
assume that $\phi \in C^4(a,b)$ has a global minimum at $c \in (a,b)$
such that $\phi'(c) = 0$ and $\phi''(c) > 0$. Then, for any $\theta \in C^2(a,b) \cap L^1(a,b)$,
we have the following asymptotic expansion:
\begin{equation}
\label{Laplace1}
I(k) := \int_{a}^{b} \theta(y) e^{-k \phi(y)} dy = \left( \frac{2\pi}{k \phi''(c)} \right)^{1/2}
\theta(c) e^{-k \phi(c)} \left[ 1 + \mathcal{O}\left(\frac{1}{k}\right) \right] \quad \mbox{\rm as} \quad
k \to \infty.
\end{equation}
\label{proposition-Laplace-1}
\end{proposition}

Here and in what follows, we use the following notations. Let $X$ be a Banach space.
We write $A = \mathcal{O}_X(k^p)$ as $k \to \infty$
if there exist constants $C_{\pm}$ such that $0 \leq C_- < C_+ < \infty$
and $C_- \mathcal{E}^p \leq \| A \|_X \leq C_+ \mathcal{E}^p$.
If $X = \R$, we write $A = \mathcal{O}(\mathcal{E}^p)$.

Applying Proposition \ref{proposition-Laplace-1}, we obtain the following.

\begin{lemma}
Under assumptions of Theorem \ref{theorem-time}, the following expansion
\begin{equation}
\label{expansion-1}
u(x,t) = k f(s_{x,a}) + \mathcal{O}_{L^{\infty}_{\rm per}(\mathbb{T})}(1) \quad
\mbox{\rm as } \quad k \to \infty,
\end{equation}
holds for all $t \in [0,t_0)$, where $t_0 := \frac{1}{2 k |f'(0)|}$
and $s_{x,a}$ is the unique root of $a (s - x) + f(s) = 0$.
\label{lemma-Laplace-1}
\end{lemma}

\begin{proof}
We compute
$$
\phi'_{x,a}(y) = f(y) + a (y - x), \quad
\phi''_{x,a}(y) = f'(y) + a.
$$
For every $a > \max_{y \in \mathbb{T}}(-f'(y)) = |f'(0)|$ and every
$x \in \mathbb{T}$, there exists exactly one root of $f(y) + a (y - x) = 0$.
Let us denote this root by $s_{x,a}$. Conditions of Proposition \ref{proposition-Laplace-1}
are satisfied, so that
$$
I_{x,a}(k) = \left( \frac{2\pi}{k (f'(s_{x,a}) + a)} \right)^{1/2}
e^{-k \phi_{x,a}(s_{x,a})} \left[ 1 + \mathcal{O}\left(\frac{1}{k}\right) \right] \quad \mbox{\rm as} \quad
k \to \infty
$$
and
\begin{eqnarray*}
\partial_x I_{x,a}(k) & = & a k \int_{-\infty}^{\infty} (y-x) e^{- k \phi_{x,a}(y)} dy\\
& = & a k (s_{x,a} - x) \left( \frac{2\pi}{k (f'(s_{x,a}) + a)} \right)^{1/2}
e^{-k \phi_{x,a}(s_{x,a})} \left[ 1 + \mathcal{O}\left(\frac{1}{k}\right) \right] \quad \mbox{\rm as} \quad
k \to \infty.
\end{eqnarray*}
Using these expansions, the representation (\ref{solution-Burgers}), and the
equation $a (x - s_{x,a}) = f(s_{x,a})$, we obtain (\ref{expansion-1}).
\end{proof}

\begin{remark}
The result of Lemma \ref{lemma-Laplace-1} is known in the limit of vanishing viscosity, when
a smooth solution of
the viscous Burgers equation is shown to converge to the classical solution of
the inviscid Burger equation before a shock is formed \cite[Chapter 2]{Whitham}.
\end{remark}

Using the representation
\begin{equation}
\label{derivative-formula}
u_x(x,t) = u^2(x,t) - \frac{\partial^2_{x} I_{x,a}(k)}{I_{x,a}(k)}
\end{equation}
and the Laplace method for $\partial^2_{x} I_{x,a}(k)$, the result of Lemma \ref{lemma-Laplace-1}
can be extended to show for all $t \in [0,t_0)$ that
\begin{equation}
\label{expansion-1-1}
u_x(x,t) = k \frac{a f'(s_{x,a})}{a + f'(s_{x,a})} +
\mathcal{O}_{L^{\infty}_{\rm per}(\mathbb{T})}(1) \quad
\mbox{\rm as } \quad k \to \infty.
\end{equation}
As a result, the energy and enstrophy are expanded as $k \to \infty$ as follows:
\begin{eqnarray*}
K(u(t)) = \frac{1}{2} k^2 \int_{\mathbb{T}} f^2(s_{x,a}) dx + \mathcal{O}(k)
\end{eqnarray*}
and
\begin{eqnarray*}
E(u(t)) = \frac{1}{2} k^2 \int_{\mathbb{T}} \frac{a^2 (f'(s_{x,a}))^2}{(a + f'(s_{x,a}))^2} dx
+ \mathcal{O}(k).
\end{eqnarray*}
Because $f(0) = f\left(\frac{1}{2}\right) = 0$,
we have $s_{x,a} = x$ for $x = 0$ and $x = \frac{1}{2}$. Moreover, the map
$\mathbb{T} \ni x \mapsto s_{x,a} \in \mathbb{T}$ is one-to-one
and onto. This argument gives the energy conservation
at the leading order for all $t \in [0,t_0)$:
\begin{eqnarray*}
K(u(t)) & = & \frac{1}{2} k^2 \int_{\mathbb{T}} f^2(y) \left(1 + \frac{1}{a} f'(y) \right) dy + \mathcal{O}(k) \\
& = & \frac{1}{2} k^2 \int_{\mathbb{T}} f^2(y) dy + \mathcal{O}(k) \\
& = & K(u_0) + \mathcal{O}(k).
\end{eqnarray*}
On the other hand, the enstrophy grows initially but remains within the $\mathcal{O}(k^2)$
order for all $t \in [0,t_0)$:
$$
E(u(t)) = \frac{1}{2} k^2 \int_{\mathbb{T}} \frac{a (f'(y))^2}{a + f'(y)}  dy
+ \mathcal{O}(k).
$$

At $t = t_0$, that is, for $a = \max_{y \in \mathbb{T}}(-f'(y)) = |f'(0)|$, a local (pitchfork) bifurcation
occurs among the roots of $f(s) + a s  = 0$ near $s = 0$. For $t > t_0$, $f'(0) + a < 0$, so that
for all $x$ near $0$, a local maximum $s_{x,a}$ of $\phi_{x,a}$ exists near $0$, whereas
two local minima $s^+_{x,a}$ and $s^-_{x,a}$ of $\phi_{x,a}$ exist for
$s^+_{x,a} \in \left(0,\frac{1}{2}\right)$ and
$s^-_{x,a} \in \left(-\frac{1}{2},0\right)$.

In what follows, we
consider the values of $x \in \left[0,\frac{1}{2}\right]$ and use
the odd symmetry of $u(x,t)$ for $x \in \left[-\frac{1}{2},0\right]$.
Note that
$\phi_{0,a}(s^+_{0,a}) = \phi_{0,a}(s^-_{0,a})$ due to the pitchfork symmetry.
When $x$ is increased in $\left(0,\frac{1}{2}\right)$, then
$\phi_{x,a}(s^+_{x,a}) < \phi_{x,a}(s^-_{x,a})$ and the difference is growing
with the values of $x$. Moreover, there is $x_0 \in \left(0,\frac{1}{2}\right)$
such that the roots $s_{x,a}$ and $s^-_{x,a}$ coalesce at $x = x_0$ and disappear
as a result of the saddle-node bifurcation. To prove these claims, we
denote
$$
\varphi_{x,a} = \phi_{x,a}(s^-_{x,a}) - \phi_{x,a}(s^+_{x,a})
= \frac{1}{2} a(s^-_{x,a}-s^+_{x,a})(s^+_{x,a} + s^-_{x,a} - 2x) - \int_{s^-_{x,a}}^{s^+_{x,a}} f(s) ds
$$
and
$$
\chi_{x,a} = \left( \frac{f'(s^+_{x,a}) + a}{f'(s^-_{x,a}) + a} \right)^{1/2}
$$
We have the following.

\begin{lemma}
For every $t > t_0$, where $t_0 := \frac{1}{2 k |f'(0)|}$,
there is $x_0 \in \left( 0, \frac{1}{2} \right)$ such that for all $x \in [0,x_0)$,
three roots of $a (s - x) + f(s) = 0$ exists in the following order:
$$
-\frac{1}{2} < s^-_{x,a} < s_{x,a} \leq 0 < s^+_{x,a} < \frac{1}{2}.
$$
Moreover, $\varphi_{x,a}$ and $\chi_{x,a}$ are $C^1$ monotonically increasing functions of $x$
with $\varphi_{0,a} = 0$, $\chi_{0,a} = 1$, and $\chi_{x,a} \to +\infty$ as $x \to x_0$.
The point $x_0$ marks the saddle-node bifurcation among
the roots of $a (s - x) + f(s) = 0$ such that $s^-_{x_0,a} = s_{x_0,a}$
and $a + f'(s_{x_0,a}) = 0$.
\label{lemma-three-roots}
\end{lemma}

\begin{proof}
The presence of three roots $s^-_{0,a} = -s^+_{0,a}$ and $s_{0,a} = 0$
follows for $x = 0$ because $a + f'(0) < 0$ for $t > t_0$ and $f$ is an odd function.
By continuity, three roots persist and, by the implicit function theorem,
the three roots are $C^1$ functions of $x$ as long as
$a + f'(s^{\pm}_{x,a}) > 0$. We now compute
$$
\partial_x \varphi_{x,a} = a (s^+_{x,a} - s^-_{x,a}) > 0,
$$
and
$$
\partial_x \chi_{x,a} = \chi_{x,a} \left[ \frac{f''(s^+_{x,a})}{(a + f'(s^+_{x,a}))^2} -
\frac{f''(s^-_{x,a})}{(a + f'(s^-_{x,a}))^2} \right] \geq 0,
$$
where the last inequality is due to $f''(x) \geq 0$ for $x \in \left[0,\frac{1}{2}\right]$
and the fact that $f''$ is an odd function. In addition, we note that
$$
\varphi_{0,a} = \frac{1}{2} a(s^-_{0,a}-s^+_{0,a})(s^+_{0,a} + s^-_{0,a}) - \int_{s^-_{0,a}}^{s^+_{0,a}} f(s) ds = 0
$$
and
$$
\chi_{0,a} = \left( \frac{f'(s^+_{0,a}) + a}{f'(s^-_{0,a}) + a} \right)^{1/2} = 1,
$$
because $s^-_{0,a} = -s^+_{0,a}$ and $f$ is odd. The statement of the lemma is proved.
\end{proof}

\begin{remark}
When $a$ is reduced further, additional saddle-node bifurcations occur among the roots of
$f(s) + a s  = 0$ due to periodicity of $f$. For all $x$ near $0$,
these bifurcations give rise to new maxima and minima of $\phi_{x,a}$
outside of $\mathbb{T}$ and the values of $\phi_{x,a}$ at new local minima are
larger than the values of $\phi_{x,a}(s^{\pm}_{x,a})$. Therefore, we can simply neglect
the presence of these additional bifurcations in the applications of the
Laplace method for all $a \in (0,|f'(0)|)$.
\end{remark}

We shall apply Proposition \ref{proposition-Laplace-1} to obtain the following.

\begin{lemma}
Let $s^{\pm}_{x,a}$ and $x_0$ be described in
Lemma \ref{lemma-three-roots}. Under assumptions of Theorem \ref{theorem-time}, for every $t > t_0$,
where $t_0 := \frac{1}{2 k |f'(0)|}$,
there is $x_1 \in (0,x_0)$ such that for all $x \in [0,x_1]$,
\begin{equation}
\label{expansion-2}
u(x,t) = k \frac{f(s^+_{x,a}) + \chi_{x,a} f(s^-_{x,a}) e^{-k \varphi_{x,a}}}{1 + \chi_{x,a} e^{-k \varphi_{x,a}}}
+ \mathcal{O}_{L^{\infty}(0,x_1)}(1) \quad
\mbox{\rm as } \quad k \to \infty
\end{equation}
whereas for all $x \in \left[x_1,\frac{1}{2}\right]$,
\begin{equation}
\label{expansion-3}
u(x,t) = k f(s^+_{x,a}) + \mathcal{O}_{L^{\infty}\left(x_1,\frac{1}{2}\right)}(1) \quad
\mbox{\rm as } \quad k \to \infty.
\end{equation}
\label{lemma-Laplace-2}
\end{lemma}

\begin{proof}
We split the integral in the explicit solution (\ref{solution-Burgers}) into two parts
$$
I_{x,a}(k) := \int_{-\infty}^0 e^{- k \phi_{x,a}(y)} dy + \int_{0}^{\infty} e^{- k \phi_{x,a}(y)} dy
$$
and apply Laplace method of Proposition \ref{proposition-Laplace-1} separately
for each integral resulting in the following
\begin{eqnarray*}
I_{x,a}(k) & = &  \left( \frac{2\pi}{k (f'(s^-_{x,a}) + a)} \right)^{1/2}
e^{-k \phi_{x,a}(s^-_{x,a})} \left[ 1 + \mathcal{O}\left(\frac{1}{k}\right) \right] \\
& \phantom{t} & + \left( \frac{2\pi}{k (f'(s^+_{x,a}) + a)} \right)^{1/2}
e^{-k \phi_{x,a}(s^+_{x,a})}
\left[ 1 + \mathcal{O}\left(\frac{1}{k}\right) \right] \quad \mbox{\rm as} \quad
k \to \infty
\end{eqnarray*}
and
\begin{eqnarray*}
\partial_x I_{x,a}(k) & = &  ak (s_{x,a}^- -x) \left( \frac{2\pi}{k (f'(s^-_{x,a}) + a)} \right)^{1/2}
e^{-k \phi_{x,a}(s^-_{x,a})} \left[ 1 + \mathcal{O}\left(\frac{1}{k}\right) \right] \\
& \phantom{t} & + ak (s_{x,a}^+ -x) \left( \frac{2\pi}{k (f'(s^+_{x,a}) + a)} \right)^{1/2}
e^{-k \phi_{x,a}(s^+_{x,a})}
\left[ 1 + \mathcal{O}\left(\frac{1}{k}\right) \right] \quad \mbox{\rm as} \quad
k \to \infty.
\end{eqnarray*}
For $x = 0$, both leading-order terms in these expansions
have equal magnitude resulting in (\ref{expansion-2}), whereas for any fixed $x > 0$,
the first integral is exponentially small compared to the second integral
resulting in (\ref{expansion-3}).
\end{proof}

\begin{remark}
The two different expansions (\ref{expansion-2}) and (\ref{expansion-3})
are asymptotically equivalent for $x = x_1$ (and near $x = x_1)$ because
$\varphi_{x,a} > 0$ for any $x > 0$ and, therefore,
the term involving $e^{-k \varphi_{x,a}}$ is exponentially small as
$k \to \infty$ for any fixed $x > 0$.
\end{remark}

Using the representation (\ref{derivative-formula})
and the Laplace method for $\partial^2_{x} I_{x,a}(k)$, the result of Lemma \ref{lemma-Laplace-2}
can be extended to prove for all $t > t_0$ and all $x \in [0,x_1]$:
\begin{equation}
\label{expansion-2-1}
u_x(x,t) = -k^2 \frac{\chi_{x,a}(f(s^+_{x,a}) - f(s^-_{x,a}))^2 e^{-k \varphi_{x,a}}}{(1 + \chi_{x,a} e^{-k \varphi_{x,a}})^2} +
\mathcal{O}_{L^{\infty}(0,x_1)}(k) \quad
\mbox{\rm as } \quad k \to \infty,
\end{equation}
whereas for all $x \in \left[x_1,\frac{1}{2}\right]$,
\begin{equation}
\label{expansion-3-1}
u_x(x,t) = \mathcal{O}_{L^{\infty}\left(x_1,\frac{1}{2}\right)}(k) \quad
\mbox{\rm as } \quad k \to \infty.
\end{equation}

We shall now compute the leading order of the enstrophy for $t > t_0$:
\begin{eqnarray}
\label{enstrophy-leading-order}
E(u(t)) = \int_0^{x_1} u_x^2(x,t) dx
+ \mathcal{O}(k^2),\quad  \mbox{\rm as } \quad k \to \infty.
\end{eqnarray}
Because $\varphi_{x,a}$ is monotonically growing for all $x \in [0,x_1]$, we can use
another version of the Laplace method (Section 3.3 in \cite{Miller}).

\begin{proposition}
For any $c > 0$, assume that $\phi \in C^2(0,c)$ is monotonically growing
such that $\phi'(0) > 0$. Then, for any $\theta \in C^1(0,c) \cap L^1(0,c)$,
we have the following asymptotic expansion:
\begin{equation}
\label{Laplace2}
I(k) := \int_{0}^{c} \theta(y) e^{-k \phi(y)} dy = \frac{\theta(0)}{k \phi'(0)}
 e^{-k \phi(0)} \left[ 1 + \mathcal{O}\left(\frac{1}{k}\right) \right] \quad \mbox{\rm as} \quad
k \to \infty.
\end{equation}
\label{proposition-Laplace-2}
\end{proposition}

We have computed previously: $\varphi_{0,a} = 0$, $\chi_{0,a} = 1$, and
$$
\partial_x \varphi_{0,a} = a (s^+_{0,a} - s^-_{0,a}) = 2 a s^+_{0,a} = -2 f(s^+_{0,a}) > 0.
$$
Therefore, a straightforward application of Proposition \ref{proposition-Laplace-2} yields
\begin{eqnarray}
\label{enstrophy-final}
E(u(t)) = \frac{1}{2} k^3 |f(s^+_{0,a})|^3
+ \mathcal{O}(k^2),\quad  \mbox{\rm as } \quad k \to \infty.
\end{eqnarray}
For $a = |f'(0)|$, $s^+_{0,a} = 0$, and the $\mathcal{O}(k^3)$ term of (\ref{enstrophy-final})
vanishes because $f(0) = 0$. When $a$ is decreased from $|f'(0)|$ to $0$,
the value of $s^+_{0,a}$ grows from $0$ to $\frac{1}{2}$ and it passes the value $x_*$, where
$f'(x_*) = 0$ and $|f(s^+_{0,a})|$ is maximal. The corresponding value of $a_*$ is found from
the equation $s^+_{0,a_*} = x_*$, or explicitly, $a_* = \frac{|f(x_*)|}{x_*}$. This argument
completes the proof of the first two bounds (\ref{arg-max-theorem}) of
Theorem  \ref{theorem-time}  with $T_* = \frac{1}{2 k a_*} = \frac{x_*}{2k |f(x_*)|}$
(recall that $k = \mathcal{O}(\mathcal{E}^{1/2})$).

\begin{remark}
An application of the Laplace method to the values of $t$ near $t_0$ that depends on $k$
is much more delicate and has been explored in the pioneer paper \cite{Chester} (see \cite{Wang} for
recent development). We are very fortunate here that the main result of Theorem \ref{theorem-time}
can be proven without knowing the behavior of the solution near $t = t_0$.
\end{remark}

To complete the proof of Theorem  \ref{theorem-time}, we need to compute the energy $K(u(t))$ for all $t > t_0$.
By Lemma \ref{lemma-Laplace-2},
we represent for all $x \in \mathbb{T}$:
$$
u(x,t) = k f(s^+_{x,a}) + \tilde{u}(x,t),
$$
where $\tilde{u}$ is found from the asymptotic expansions (\ref{expansion-2}) and (\ref{expansion-3}).
By Proposition \ref{proposition-Laplace-2}, we have
\begin{eqnarray*}
K(u(t)) =  k^2 \int_{0}^{1/2} f^2(s^+_{x,a}) dx + \mathcal{O}(k) \quad  \mbox{\rm as } \quad k \to \infty.
\end{eqnarray*}

For all $a < |f'(0)|$, we have $s^+_{0,a} > 0$. Nevertheless, we still have
$s^+_{x,a} = x$ for $x = \frac{1}{2}$. The map
$\left[0,\frac{1}{2}\right] \ni x \mapsto s^+_{x,a} \in \left[s^+_{0,a},\frac{1}{2}\right]$ is one-to-one
and onto. As a result, for all $t > t_0$:
\begin{eqnarray*}
K(u(t)) & = & k^2 \int_{s^+_{0,a}}^{1/2} f^2(y) \left(1 + \frac{1}{a} f'(y) \right) dy + \mathcal{O}(k) \\
& = & k^2 \left( \int_{s^+_{0,a}}^{1/2} f^2(y) dy + \frac{1}{3a} |f(s^+_{0,a})|^3 \right) + \mathcal{O}(k).
\end{eqnarray*}
This argument completes the proof of the third bound (\ref{arg-max-theorem}) of
Theorem  \ref{theorem-time}. To prove the bound (\ref{sharp-bound-K}), we need to show that
the $\mathcal{O}(k^2)$ term in the expansion for $K(u(t))$ is different from the one for $K(u_0)$
(it can only be smaller), or equivalently, that
\begin{equation}
\label{required-bound}
\int_{0}^{s^+_{0,a}} f^2(y) dy > \frac{1}{3a} |f(s^+_{0,a})|^3 = \frac{1}{3} s^+_{0,a} f^2(s^+_{0,a}).
\end{equation}
To prove (\ref{required-bound}), we define two functions $F,H : [0,x_*] \to \R$ by
$$
G(x) := \int_0^x f^2(y) dy - \frac{1}{3} x f^2(x), \quad H(x) = f(x) - x f'(x).
$$
By the conditions on $f$, it is clear that $G \in C^3([0,x_*])$ and $H \in C^2([0,x_*])$. Furthermore,
$G(0) = 0$, $H(0) = 0$, and
$$
G'(x) = \frac{2}{3} f(x) H(x), \quad H'(x) = -x f''(x).
$$
Because $f''(x) \geq 0$ for all $x \in \left[0,\frac{1}{2}\right]$,
$H$ is a monotonically decreasing function from $H(0) = 0$ to $H(x) < 0$ for all $x \in (0,x_*]$.
Therefore, $G$ is a monotonically increasing function from $G(0) = 0$ to $G(x) > 0$ for all $x \in (0,x_*]$.
Therefore, the inequality (\ref{required-bound}) is proved and hence, the
bound (\ref{sharp-bound-K}) is verified. Note that (\ref{enstrophy-final}) and
(\ref{required-bound}) yield explicit bounds (\ref{bound-1}) and (\ref{bound-2}). The proof of
Theorem \ref{theorem-time} is complete.

{\bf Acknowledgements:} The author thanks C. Doering for encouragements to write this paper,
J. Goodman for the stimulating idea of how to obtain the main result, P. Miller for
useful discussions of details of the Laplace method, and W. Craig for careful reading
of the manuscript. The research was supported in part by the NSERC
and was completed during author's visit to the University of Michigan at Ann Arbor.

\end{document}